\documentclass[11pt,twoside]{amsart}
\usepackage{amsmath, amsthm, amsfonts, amssymb}
\usepackage[numbers, sort & compress]{natbib}
\usepackage[mathcal,mathscr]{eucal}
\usepackage{mathrsfs}
\usepackage{graphicx, xcolor}
\usepackage[colorlinks]{hyperref}

\newtheorem{theorem}{Theorem}[section]
\newtheorem{lemma}[theorem]{Lemma}

\newtheorem{corollary}[theorem]{Corollary}
\theoremstyle{definition}
\newtheorem{definition}[theorem]{Definition}

\newtheorem{problem}[theorem]{Problem}

\begin{document}

\title[\scriptsize{Generalized Cartesian Symmetry Classes}]{Generalized Cartesian Symmetry Classes}

\author[S. S. Gholami and Y. Zamani]{Seyyed Sadegh Gholami and Yousef Zamani$^{\ast }$}

\address{Department of Mathematics, Faculty of Basic Sciences, Sahand University of Technology, Tabriz, Iran}
\email{rgolamie@yahoo.com}

\address{Department of Mathematics, Faculty of Basic Sciences, Sahand University of Technology, Tabriz, Iran}
\email{zamani@sut.ac.ir}

\thanks{$^{\ast}$ Corresponding author}

\maketitle

\begin{abstract}
Let $V$ be a unitary space. Suppose $G$ is a subgroup of the full symmetric group $S_{m}$ and $\mathfrak{X}$ is an irreducible unitary representation of $G$. In this paper, we introduce the generalized Cartesian symmetry class over $V$ associated with $G$ and $\mathfrak{X}$. Then we investigate some important properties of this vector space. Also, we study some basic properties of the induced linear operators on the generalized Cartesian symmetry classes. Some open problems are also given.
\end{abstract}
\hspace{-1cm}\rule{\textwidth}{0.2mm}
\section{\bf Introduction and Preliminaries}
In recent years, the study of symmetry classes has played a fundamental role in various branches of mathematics (see \cite{Babaei 1,Babaei 2,Gholami,Holmes 1,Holmes 2,Merris,Zamani,Zamani 1}). In this paper, we focus on the generalized Cartesian symmetry class associated with an irreducible unitary representation of a subgroup of the full symmetric group. Our main goal is to establish important properties of this vector space.\\

Let $S_m$ denote the full symmetric group of degree $m$, and let $G$ be a subgroup of $S_m$. Let $U$ be a unitary space, meaning a finite dimensional complex vector space equipped with an inner product. The set of all linear operators on $U$ is denoted by $End~(U)$. Assume that $\mathfrak{X}$ is an irreducible unitary representation of $G$ over $U$.
The generalized trace function
 $Tr_\mathfrak{X}:\mathbb{C}_{m\times m}\longrightarrow \text{End}~(U)$ is defined by
  $$
  Tr_\mathfrak{X} (A)=\sum_{\sigma\in G}\mathfrak{X}(\sigma)\sum_{i=1}^m a_{i\sigma(i)}
  $$
  for  $A=(a_{ij})\in \mathbb{C}_{m\times m}$.\\
It is proved that $Tr_\mathfrak{X}(A^{\ast})=Tr_\mathfrak{X}(A)^{\ast}$.  In particular,
if $A$ is Hermitian, then $Tr_\mathfrak{X}(A)$ is Hermitian (see \cite{Lei2}).\\

Let $V$ be a unitary space of dimension $n$ and denote by   $\times^{m}V$ be the Cartesian product of $m$-copies of $V$. Then $U\otimes V^{\times m}$ is a unitary space with an induced inner product given by
$$
 \langle u\otimes x^\times ,  v\otimes y^\times\rangle=\langle u,v \rangle \sum_{i=1}^m  \langle x_i,y_i \rangle,
$$
where $u, v\in U$ and $x^\times=(x_1,\cdots, x_m)$, $y^\times=(y_1,\cdots, y_m)\in \times^{m}V$.\\

The {\em generalized Cartesian symmetrizer associated with $G$ and $\mathfrak{X}$} is defined by
$$
C_\mathfrak{X}=\frac{1}{|G|}\sum_{\sigma\in G}\mathfrak{X}(\sigma)\otimes Q(\sigma),
 $$
 where
\begin{equation*}
 Q(\sigma)(v_1,\cdots , v_m )=(v_{\sigma^{-1}(1)},\cdots, v_{\sigma^{-1}(m)})
 \end{equation*}
 is Cartesian permutation operator with respect to $\sigma\in G$.\\

In the following theorem we show that $C_{\mathfrak{X}}$ is an orthogonal projection on $U\otimes \times^{m}V$.

\begin{theorem}\label{projection}
Suppose $\mathfrak{X}$  is an irreducible unitary representation of $G$ over unitary space $U$. Then $C_\mathfrak{X}$  is an orthogonal projection on $U\otimes \times^{m}V$.
\end{theorem}
\begin{proof}
We first prove that $C_\mathfrak{X}$ is Hermitian. We have
\begin{align*}
(C_\mathfrak{X})^*=&\left(\frac{1}{|G|}\sum_{\sigma\in G}\mathfrak{X}(\sigma)\otimes Q(\sigma)\right)^*\\=&\frac{1}{|G|}\sum_{\sigma\in G}\mathfrak{X}(\sigma)^*\otimes Q(\sigma)^*\\
=&\frac{1}{|G|}\sum_{\sigma\in G}\mathfrak{X}(\sigma^{-1})\otimes Q(\sigma^{-1})\\
=&C_\mathfrak{X}.
\end{align*}
Now we show that $C_{\mathfrak{X}}$ is idempotent. We have
\begin{align*}
(C_\mathfrak{X})^2=&\left(\frac{1}{|G|}\sum_{\sigma\in G}\mathfrak{X}(\sigma)\otimes Q(\sigma)\right)\left(\frac{1}{|G|}\sum_{\pi\in G}\mathfrak{X}(\sigma)\otimes Q(\pi)\right)\\
=&\frac{1}{|G|^2}\sum_{\sigma\in G}\sum_{\pi\in G}\mathfrak{X}(\sigma)\mathfrak{X}(\pi)\otimes Q(\sigma) Q(\pi)\\
=&\frac{1}{|G|^2}\sum_{\sigma\in G}\sum_{\pi\in G}\mathfrak{X}(\sigma\pi)\otimes Q(\sigma\pi)\hspace{10mm}(\sigma\pi=\tau)\\
=&\frac{1}{|G|^2}\sum_{\sigma\in G}\sum_{\tau\in G}\mathfrak{X}(\tau)\otimes Q(\tau)\\
=&\frac{1}{|G|}\sum_{\sigma\in G}C_\mathfrak{X}\\
=& C_\mathfrak{X}.
\end{align*}
\end{proof}
\begin{definition}
The range of $C_{\mathfrak{X}}$,\\
 $$
V^{\mathfrak{X}}(G):= C_\mathfrak{X}(U\otimes\times^{m}V),
 $$
is called the {\em generalized Cartesian symmetry class} over $V$ associated with $G$ and $\mathfrak{X}$.
\end{definition}
If $\dim U = 1$, then $V^\mathfrak{X}(G)$ reduces to $V^\chi(G)$, which is the Cartesian symmetry class associated with $G$ and the irreducible character $\chi$ of $G$ corresponding to the representation $\mathfrak{X}$ (see \cite{Gholami,Lei, Zamani}).
The elements of $V^\mathfrak{X}(G)$ that have the form
$C_\mathfrak{X}(u\otimes x^\times) $
are called the generalized Cartesian symmetrized vectors. The equality of two
generalized symmetrized vectors has been studied in \cite{Lei2}.
We will need the following theorem (see \cite[Corollary 5.9]{Lei2}).
\begin{theorem}\label{Corollary 5.9}
Let $\mathfrak{X}$ be a unitary representation of $G$ over unitary space $U$ and
$x^{\times}, y^{\times}\in \times^{m}V$. Let $A=[a_{ij}], B=[b_{ij}]\in \mathbb{C}_{m\times n}$ such that $x_{i}=\sum_{j=1}^{n}a_{ij}e_{j}, y_{i}=\sum_{j=1}^{n}b_{ij}, i=1,\cdots,m$. Then the following are equivalent:
\begin{itemize}
\item[(a)]
$C_{\mathfrak{X}}(u\otimes x^{\times})=C_{\mathfrak{X}}(u\otimes y^{\times})$ for all $u\in U$.
\item[(b)]
$Tr_{\mathfrak{X}}(AA^{\ast})=Tr_{\mathfrak{X}}(AB^{\ast})=Tr_{\mathfrak{X}}(BB^{\ast}).$
\end{itemize}
\end{theorem}
The following theorem states the inner product two generalized symmetrized vectors in terms of the generalized trace function.
\begin{theorem}\label{inner}
For all $u,v\in U$ and $x^\times,y^\times\in \times^{m}V$ we have
$$
\left\langle C_\mathfrak{X}(u\otimes x^\times),v\otimes y^\times \right\rangle=\dfrac{1}{|G|}\langle Tr_\mathfrak{X}(A)u,v\rangle,
$$
where $A=[a_{ij}]\in \mathbb{C}_{m\times m}$ and $a_{ij}=\langle x_i,y_j\rangle$.
\end{theorem}
\begin{proof}
From definition $C_{\mathfrak{X}}$,
\begin{align*}
\left\langle C_{\mathfrak{X}}(u\otimes x^\times),v\otimes y^\times \right\rangle=&\langle\dfrac{1}{|G|}\sum_{\sigma\in G}\mathfrak{X}(\sigma) u\otimes Q(\sigma)x^\times ,v\otimes y^\times\rangle\\
=&\dfrac{1}{|G|}\sum_{\sigma\in G}\langle \mathfrak{X}(\sigma) u,v\rangle\sum_{i=1}^m \langle x_{\sigma^{-1}(i)},y_{i}\rangle\\
=&\dfrac{1}{|G|}\sum_{\sigma\in G}\langle\sum_{i=1}^m\langle x_{\sigma^{-1}(i)},y_i\rangle \mathfrak{X}(\sigma)u,v\rangle\\
=&\dfrac{1}{|G|}\langle\sum_{\sigma\in G}\mathfrak{X}(\sigma)\sum_{i=1}^m \langle x_i,y_{\sigma(i)}\rangle u,v\rangle\\
=&\dfrac{1}{|G|}\langle Tr_\mathfrak{X}(A)u,v\rangle,
\end{align*}
and the result holds.
\end{proof}

In this paper, we will refer to the following lemma frequently.
\begin{lemma}\label{1}
Let $\sigma\in G$, $u\in U$ and $x^\times \in \times^{m}V$. Then
  $$
  C_\mathfrak{X}(u\otimes x^\times_{\sigma})=C_\mathfrak{X}(\mathfrak{X}(\sigma)u \otimes x^\times).
  $$
\end{lemma}
\begin{proof}
From definition $C_\mathfrak{X}$, We have
\begin{align*}
  C_\mathfrak{X}(u\otimes x^\times_{\sigma})=& \frac{1}{|G|}\sum_{\tau\in G}\left( \mathfrak{X}(\tau)\otimes Q(\tau)\right)(u\otimes x^\times_{\sigma})\\
=& \frac{1}{|G|}\sum_{\tau\in G}^{}\mathfrak{X}(\tau) u\otimes Q(\tau) x^\times_{\sigma}\\
=&\frac{1}{|G|}\sum_{\tau\in G}^{}\mathfrak{X}(\tau) u\otimes Q(\tau) Q(\sigma^{-1})x^\times
\\
=& \frac{1}{|G|}\sum_{\tau\in G}^{}\mathfrak{X}(\tau) u\otimes Q(\tau\sigma^{-1})x^\times \hspace{10mm} (\tau\sigma^{-1}=\pi )\\
=&  \frac{1}{|G|}\sum_{\pi\in G}^{}\mathfrak{X}(\pi \sigma) u\otimes Q(\pi)x^\times   \hspace{15mm} \\
=&   \frac{1}{|G|}\sum_{\pi\in G}^{}\mathfrak{X}(\pi)\mathfrak{X}(\sigma) u\otimes Q(\pi)x^\times\\
=& \left( \frac{1}{|G|}\sum_{\pi\in G}^{}\mathfrak{X}(\pi )\otimes Q(\pi)\right) (\mathfrak{X}(\sigma) u\otimes x^\times)\\
=& C_\mathfrak{X}(\mathfrak{X}(\sigma) u \otimes x^\times).
\end{align*}
\end{proof}
\begin{definition}
Suppose $G_{p}$ is the stabilizer subgroup of $p$ where $p=1,2, \cdots , m$.
The linear map
 $ T_{p} :U\longrightarrow U $
defined by
\[T_{p} =\frac{1}{|G_{p} |} \sum_{\sigma\in G_{p} }^{}\mathfrak{X}(\sigma)\]
is called the linear map corresponding to $p$.
\end{definition}

\begin{theorem}\label{T_p}
\begin{itemize}
\item[(a)] The linear map
$T_{p}$
is an orthogonal projection on $U$.
\item[(b)]
$rank~T_p=\frac{1}{|G_p|}\sum_{\sigma\in G_p}\chi(\sigma)$, where $\chi$ 
is the irreducible character of $G$ corresponding to the representation $\mathfrak{X}$.
\end{itemize}
\end{theorem}
\begin{proof}
(a) We first prove that $T_{p}$ is Hermitian. We have
$$
T_{p}^{\ast}=\left(\frac{1}{|G_p|}\sum_{\sigma\in G_p}^{}\mathfrak{X}(\sigma)\right)^{\ast}\\
=\frac{1}{|G_p|}\sum_{\sigma\in G_{p}}\mathfrak{X}(\sigma)^{\ast}\\
=\frac{1}{|G_p|}\sum_{\sigma\in G_{p}}\mathfrak{X}(\sigma^{-1})\\
= T_{p}\hspace{1mm} .
$$
 Now we show that $T_{p}$ is idempotent. We have
\begin{align*}
T_{p }^2 =&\left(\frac{1}{|G_p|}\sum_{\sigma\in G_{p}}\mathfrak{X}(\sigma)\right)\left(\frac{1}{|G_p|}\sum_{\pi\in G_{p}}\mathfrak{X}(\pi)\right)\\=
&\frac{1}{|G_p|^2}\sum_{\sigma\in G_{p}}\sum_{\pi\in G_{p}}\mathfrak{X}(\sigma)\mathfrak{X}(\pi)\\
&\frac{1}{|G_p|^2}\sum_{\sigma\in G_{p}}\sum_{\pi\in G_{p}}\mathfrak{X}(\sigma\pi)\\
=&\frac{1}{|G_p|^2}\sum_{\sigma\in G_{p}}\sum_{\tau\in G_{p}}\mathfrak{X}(\tau)\hspace{10mm}(\sigma\pi=\tau)\\
=&\frac{1}{|G_p|^2}\sum_{\sigma\in G_{p}}|G_p| T_{p}\\
=&T_{p} \; .
 \end{align*}
(b)
$$
\text{rank}~T_p=tr~(T_p)=tr \left[\frac{1}{|G_p|}\sum_{\sigma\in G_p}\mathfrak{X}(\sigma)\right]=\frac{1}{|G_p|}\sum_{\sigma\in G_p}\chi (\sigma).
$$
\end{proof}


In this paper, we study some important properties of the vector space $V^\mathfrak{X}(G)$.
\section{\bf The generalized Cartesian Symmetry Classes}
Suppose $\mathbb{F}=\lbrace u_1,\cdots,u_r\rbrace $ and $\mathbb{E}=\lbrace e_1,\cdots,e_n \rbrace $
are orthonormal bases for unitary
spaces $U$ and $V$, respectively. Assume $[\mathfrak{X}(\sigma)]_{\mathbb{F}}=[m_{ij}(\sigma)]$ for any $\sigma \in G$. For $1\leq i\leq n$ and $1\leq j\leq m$, let
$$
e_{ij}=(\delta_{1j}e_{i},\delta_{2j}e_{i},\cdots , \delta_{mj}e_{i}) \in \times^{m}V.
$$
Then the set
$$
\mathbb{B}=\lbrace u_{k}\otimes e_{ij}~|~1\leq k\leq r,  1\leq i \leq n ,1\leq j \leq m \rbrace
$$
is an orthonormal basis of $U\otimes \times^{m}V$.
 Therefore,
 $$
 V^\mathfrak{X}(G) = \langle C_\mathfrak{X}(u_k\otimes e_{ij}) \mid  1\leq k\leq r, 1\leq i\leq n, 1\leq j\leq m \rangle.
 $$
The elements 
$$
C_\mathfrak{X}(u_k\otimes e_{ij}),~ 1\leq k\leq r,~ 1\leq i\leq n,~ 1\leq j\leq m 
$$
of $V^{\mathfrak{X}}(G)$ are called {\em the generalized Cartesian standard symmetrized vectors}.
 \begin{definition}
For any
$ 1\leqslant j, s \leqslant m$, we define the linear map
 $ T_{sj} :U\longrightarrow U $
by
\[
T_{sj} =\frac{1}{|G_{sj} |} \sum_{\sigma\in G_{sj} }^{}\mathfrak{X}(\sigma),
\]
where
 $$
G_{sj}=\lbrace{\sigma \in G}~|~{\sigma(j)} =s \rbrace.
$$
If
$G_{sj}$ is empty, then we define  $T_{sj }=0$.
 If
 $s=j$, then $G_{jj}=G_j$, the stabilizer of $j$ in $G$ and so $T_{jj}=T_j$, the linear map corresponding to $j$.
\end{definition}

 \begin{theorem}\label{orthogonal}
 For any
 $1\leqslant j,s \leqslant m, 1\leqslant i,r \leqslant n,   1\leqslant k,l  \leqslant r$,
we have
\begin{equation*}
\langle C_{\mathfrak{X}}(u_k\otimes e_{ij}), C_{\mathfrak{X}}(u_l\otimes e_{rs})\rangle = \left\{ \begin{array}{ll} 0 &  s \nsim j \\
\delta_{ir} \dfrac{|G_{sj}|}{|G|}\langle T_{sj}u_k , u_l\rangle &  s \sim j
\end{array}\right.
\end{equation*}
\vspace{0.5 cm}
     In particular,
     $$
     \parallel C_\mathfrak{X}(u_k\otimes e_{ij})\parallel^2= \frac{1}{[G:G_j]}\parallel T_j u_k\parallel^2.
     $$
\end{theorem}

\begin{proof}
By using Theorems \ref{projection} and \ref{inner}, we have
\begin{align*}
\langle C_\mathfrak{X}(u_k\otimes e_{ij}) , C_\mathfrak{X}(u_l\otimes e_{rs})\rangle=&\langle C_\mathfrak{X}(u_k\otimes e_{ij}) , u_l\otimes e_{rs}\rangle
 \\
 =&\dfrac{1}{|G|}\langle Tr_\mathfrak{X}(A)u_k, u_l\rangle \\
 =&\dfrac{1}{|G|}\langle\sum_{\sigma\in G}\mathfrak{X}(\sigma){\sum_{p=1}^{m}} a_{p\sigma( p)} u_k , u_l\rangle ,\\
\end{align*}
where
 $$
a_{pq}=\langle \delta_{pj}e_{i}, \delta_{qs}e_r\rangle=\delta_{pj}\delta_{qs}\langle e_i, e_r\rangle=\delta_{pj}\delta_{qs}\delta_{ir}.
$$
Therefore
 \begin{align*}
\langle C_\mathfrak{X}(u_k\otimes e_{ij}) , C_\mathfrak{X}(u_l\otimes e_{rs})\rangle =&\dfrac{1}{|G|}\langle\sum_{\sigma\in G}\mathfrak{X}(\sigma){\sum_{p=1}^{m}}\delta_{pj}\delta_{\sigma (p)s}\delta_{ir} u_k,u_l\rangle\\
 =&\delta_{ir}\dfrac{1}{|G|}\langle\sum_{\sigma\in G}\mathfrak{X}(\sigma)\delta_{\sigma(j)s} u_k,u_l\rangle\\
 =&\left\{\begin{array}{ll} 0 & s \nsim j\\ 
 \delta_{ir}\dfrac{1}{|G|}\langle\sum_{\sigma\in G_{sj}}\mathfrak{X}(\sigma) u_k,u_l\rangle& s \sim j\end{array}\right.\\ 
=&\left\{\begin{array}{ll} 0 & s \nsim j\\ 
\delta_{ir} \dfrac{|G_{sj}|}{|G|}\langle T_{sj}u_k ,u_l\rangle& s \sim j\end{array}\right.\\ 
\end{align*}
  In particular
\begin{align*}
\parallel C_\mathfrak{X}(u_k\otimes e_{ij})\parallel^2=&\langle C_\mathfrak{X}(u_k\otimes e_{ij}) , C_\mathfrak{X}(u_k\otimes e_{ij})\rangle\\
 =&\dfrac{|G_{j}|}{|G|}\langle T_{j}u_k ,u_k\rangle\\
 =&\dfrac{1}{[G:G_{j}]}\langle T_{j} u_k , T_{j} u_k\rangle  \hspace{15mm}  (T_{j}^{2}=T_{j}=T_{j}^{\ast})\\
 =&\dfrac{1}{[G:G_{j}]} \parallel T_ju_k \parallel^2.
\end{align*}
\end{proof}
From the above Theorem, we deduce that $C_\mathfrak{X}(u_k\otimes e_{ij}) = 0$ if and only if $T_{j}u_k = 0$.
For any $1\leq k\leq r$, let
$$
\Omega_{k}=\{1\leq j\leq m ~|~ T_{j}u_{k} \neq 0 \}.
$$
Put
$\Omega=\bigcup_{k=1}^{r}\Omega_{k}$. Then $\Omega=\{1\leq j\leq m ~|~T_{j}\neq 0\}$.
By Theorem \ref{T_p}, $T_j\neq 0$ if and only if $\sum_{\sigma\in G_j}\chi (\sigma)\neq 0$.
Hence 
$$
\Omega=\{1\leq j\leq m\mid \sum_{\sigma\in G_j}\chi (\sigma)\neq 0 \}=\{1\leq j\leq m\mid [\chi , 1_{G_j}]\neq 0 \},
 $$
 where $[\hspace{1.5mm},\hspace{1.5mm}]$ is the inner product of characters (see \cite{Isaacs}).

 Let $\bar{\mathcal{D}}=\mathcal{D} \cap \Omega$. For each $1\leq j\leq m $ and $ 1\leq i\leq n$, the subspace
$$
 V^\mathfrak{X}_{ij}(G) = \langle C_\mathfrak{X}(u_k\otimes e_{ij}) \mid  1\leq k\leq r  \rangle
 $$
   is called the {\em  generalized cyclic  subspace}.
   If $\text{dim}\ U=1$, then $ V^\mathfrak{X}_{ij}(G)$ reduces to $ V^{\chi}_{ij}(G)$, the cyclic subspace associated with $G$ and the irreducible character $\chi$ of $G$ (see \cite{Gholami, Zamani}). \\


Since
$ \langle \mathfrak{X}(\sigma) u_1: \sigma \in G  \rangle$
is a non-zero submodule of the irreducible $C[G]$-module $U$, so
$ \langle \mathfrak{X}(\sigma) u_1: \sigma \in G  \rangle =U$. Therefore it is to see that for every $1\leq j\leq m $ and $ 1\leq i\leq n$,
 $$
 V^\mathfrak{X}_{ij}(G) = \langle C_\mathfrak{X}(u_1\otimes e_{i\sigma(j)}) \mid  \sigma \in G  \rangle.
 $$
 For each $ 1\leq i\leq n$, we define
$$
 V^{\mathfrak{X}}_{i}(G) = \langle C_\mathfrak{X}(u_k\otimes e_{ij}) \mid  1\leq k\leq r, 1\leq j\leq m  \rangle.
 $$
 By Theorem \ref{orthogonal}, if $i\neq r$ then $V^\mathfrak{X}_{i}(G) \perp V^\mathfrak{X}_{r}(G)$. Thus
 $V^\mathfrak{X}(G) =\bigoplus_{i=1}^{n}V^\mathfrak{X}_{i}(G)~(\text{orthogonal})$.
 For $1\leq j, s\leq m$, if $j\sim s$ then by Lemma \ref{1}, $V^\mathfrak{X}_{ij}(G)=V^\mathfrak{X}_{is}(G)$, otherwise $V^\mathfrak{X}_{ij}(G)\perp V^\mathfrak{X}_{is}(G)$, by Theorem \ref{orthogonal}.
 Hence
 $$
V^{\mathfrak{X}}_{i}(G)=\bigoplus_{j\in\bar{\mathcal{D}}}V^\mathfrak{X}_{ij}(G)~(\text{orthogonal}).
$$
Therefore
$$
V^{\mathfrak{X}}(G)=\bigoplus_{i=1}^{n}\bigoplus_{j\in\bar{\mathcal{D}}}V^{\mathfrak{X}}_{ij}(G)~(\text{orthogonal}).
$$
The following theorem provides a formula for computing the dimension of the generalized cyclic subspace.
\begin{theorem}
 Let $\mathfrak{X}$ be an irreducible unitary representation of $G$ over a unitary space $U$. Suppose $\mathfrak{X}$ affords the irreducible character $\chi$ of $G$. If $j\in\bar{\mathcal{D}}$ then
$$ 
\dim\hspace{1mm}V^{\mathfrak{X}}_{ij}(G)=[\chi , 1_{G_{j}}].
$$
\end{theorem}
\begin{proof}
Let $j\in\bar{\mathcal{D}}$, $[G:G_j]=t$ and  $G=\bigcup_{i=1}^{t}\sigma_{i} G_j$,
be the left coset decomposition of $G_{j}$ in  $G$. Then $|Orb_{G}~(j)|=t$.
Suppose
$$
Orb_{G}~(j)=\lbrace \sigma_{1}(j), \cdots, \sigma_{t}(j) \rbrace .
$$
Notice that
\begin{align*}
V^\mathfrak{X}_{ij}(G)=C_{\mathfrak{X}}(W_{ij}),
\end{align*}
where
\begin{align*}
W_{ij}=\left\langle u_k\otimes e_{i \sigma (j)}\mid 1\leq k\leq r,~\sigma\in G\right\rangle .
\end{align*}
Then
\begin{align*}
\mathbb{E}_{ij}=\{ u_k\otimes e_{i \sigma_{s} (j)} \mid 1\leq k\leq r,~1\leq s\leq t\}
\end{align*}
 is a basis of $W_{ij}$  but the set $C_{\mathfrak{X}}(\mathbb{E}
_{ij})$ may not be a basis for $V^\mathfrak{X}_{ij}(G)$.
Since $W_{ij}$ is an invariant subspace of $C_\mathfrak{X}$, so the restriction
 $C_\mathfrak{X} \mid _{W_{ij}}=C_{\mathfrak{X}}^{ij} :  W_{ij} \to W_{ij}$
 is a linear operator.
 We put
 $$
  [C^{ij}_{\mathfrak{X}}]_{\mathbb{E}_{ij}}=B=[b_{(k,l),(p,q)}].
 $$
Now we have
 \begin{align*}
  C^{ij}_\mathfrak{X}(u_{p}\otimes  e_{i\sigma_{q}(j)})=&C_\mathfrak{X}(u_{p}\otimes  e_{i\sigma_{q}(j)})\\
  =&C_\mathfrak{X}(\mathfrak{X}(\sigma_{q}^{-1})u_p\otimes e_{ij})\\
 =&\frac{1}{|G|}\sum_{\sigma\in G}\mathfrak{X}(\sigma\sigma_{q}^{-1})u_p\otimes  e_{i\sigma(j)}\\
=&\frac{1}{|G|}\sum_{l=1}^t\left(\sum_{\sigma\in \sigma_{l}{G_j}}\mathfrak{X}(\sigma \sigma_{q}^{-1})u_p\otimes e_{i\sigma(j)}\right)\\
 \hspace{60mm}=&\frac{1}{|G|}\sum_{l=1}^t(\sum_{\tau\in G_{j}}\mathfrak{X}(\sigma_l\tau\sigma_{q}^{-1})u_p\otimes  e_{i\sigma_{l}(j)})\\
  =&\frac{1}{|G|}\sum_{l=1}^t\sum_{\tau\in G_j}\sum_{k=1}^{r} m_{kp}(\sigma_l\tau\sigma_{q}^{-1})u_{k}\otimes  e_{i\sigma_{l}(j)}\\
  =&\sum_{l=1}^t\sum_{k=1}^r\left[\frac{1}{|G|}\sum_{\tau\in G_j} m_{kp}(\sigma_l\tau\sigma_{q}^{-1})\right]u_{k}\otimes  e_{i\sigma_{l}(j)} .
\end{align*}
So
 $$
 b_{(k,l),(p,q)}=\frac{1}{|G|}\sum_{\tau\in G_j} m_{kp}(\sigma_l\tau\sigma_{q}^{-1}).
 $$
 We prove that  $B$ is an idempotent matrix. We have
\begin{align*}
(B^2)_{(k,l),(k^\prime,l^\prime)}=&\sum_{p=1}^r\sum_{q=1}^t b_{(k,l),(p,q)} b_{(p,q),(k^\prime,l^\prime)}\\
=&\sum_{p=1}^r\sum_{q=1}^t \left(\frac{1}{|G|}\sum_{\tau\in G_j}m_{kp}(\sigma_l\tau\sigma_{q}^{-1})\right)\left(\frac{1}{|G|}\sum_{\mu \in G_j}m_{pk^\prime}(\sigma_q\mu\sigma_{l^\prime}^{-1})\right)\\
=&\frac{1}{|G|^2}\sum_{p=1}^r\sum_{q=1}^t\sum_{\tau\in G_j}\sum_{\mu\in G_j} m_{kp}(\sigma_l\tau\sigma_q^{-1})m_{pk^\prime}(\sigma_q\mu\sigma_{l^\prime}^{-1})\\
=&\frac{1}{|G|^2}\sum_{\mu,\tau\in G_j}\sum_{q=1}^t m_{kk^\prime}(\sigma_l\tau\mu\sigma_{l^\prime}^{-1})\\
=&\frac{t|G_j|}{|G|^2}\sum_{g\in G_j}  m_{kk^\prime}(\sigma_lg\sigma_{l^\prime}^{-1})~~~~(g=\tau\mu)\\
=&\frac{1}{|G|}\sum_{g\in G_j}  m_{kk^\prime}(\sigma_lg\sigma_{l^\prime}^{-1})\\
=&B_{(k,l),(k^\prime,l^\prime)}.
\end{align*}
Thus
\[\text{dim}\hspace{1mm} V_{ij}^\mathfrak{X}(G)=\text{rank}\hspace{1mm} C_{\mathfrak{X}}^{ij}=\text{rank}\hspace{1mm} B=\text{tr}\hspace{1mm} B\]
Now we calculate  $\text{tr}\hspace{1mm} B$. We have
\begin{align*}
\text{tr}\hspace{1mm}B=&\sum_{k=1}^r\sum_{l=1}^t b_{(k,l),(k,l)}\\
=&\sum_{k=1}^r\sum_{l=1}^t\left(\frac{1}{|G|}\sum_{\tau\in G_j} m_{kk}(\sigma_l\tau\sigma_l^{-1})\right) \\
=&\frac{1}{|G|}\sum_{\tau \in G_j}\sum_{l=1}^t\sum_{k=1}^r m_{kk}(\sigma_l\tau\sigma_l^{-1})\hspace{10mm}\\
=&\frac{1}{|G|}\sum_{\tau\in G_j}\sum_{l=1}^t \chi(\sigma_l\tau\sigma_l^{-1})\\
=&\frac{1}{|G|}\sum_{\tau\in G_j}\sum_{l=1}^t \chi(\tau)\\
=&\frac{t}{|G|}\sum_{\tau\in G_j}\chi(\tau)\hspace{10mm}([G:G_j)]=t)\\
=&\frac{1}{|G_j|}\sum_{\tau\in G_j}\chi(\tau)\\
=&[\chi,1_{G_j}],
\end{align*}
so the result holds.
\end{proof}
Now we construct a basis for the generalized Cartesian symmetry class $V^\mathfrak{X}(G)$. Since
$V^{\mathfrak{X}}(G)=\bigoplus_{i=1}^{n}\bigoplus_{j\in \overline{D}}^{}V^\mathfrak{X}_{ij}(G)$, in order to find a basis for $V^{\mathfrak{X}}(G)$, it suffices to find a basis for the generalized
cyclic subspace $V^\mathfrak{X}_{ij}(G)$ for every $1\leq i\leq n$ and $j\in \bar{\mathcal{D}}$.
Let $j\in \overline{\mathcal{D}}$ and $\dim V^\mathfrak{X}_{ij}(G)=s_j$. Since
$$
V^\mathfrak{X}_{ij}(G) = \langle C_\mathfrak{X}(u_1\otimes e_{i\sigma(j)}) \mid  \sigma \in G  \rangle,
$$
 so we can choose the ordered subset
 $
\{j_{1}, \cdots, j_{s_{j}}\}
$
from the orbit of $j$, such that the set
$$
\{C_\mathfrak{X}(u_1\otimes e_{ij_{1}}), \cdots ,  C_\mathfrak{X}(u_1\otimes e_{ij_{s_j}})\}
$$
is a basis for the generalized cyclic subspace $V^\mathfrak{X}_{ij}(G)$.
 Execute this procedure for each $k\in\bar{\mathcal{D}}$.
If $\bar{\mathcal{D}} =\{j, k, l, \cdots  \}\ (j<k<l< \cdots)$,  take
$$
\hat{\mathcal{D}}=\{j_{1}, \cdots, j_{s_{j}}; k_{1},\cdots, k_{s_{k}}; \cdots\}
$$
 to be ordered as indicated.
 Then
 $$
\{ C_\mathfrak{X}(u_1\otimes e_{ij})~|~1\leq i\leq n,\  j\in \hat{\mathcal{D}}\}
$$
is a basis of $V^\mathfrak{X} (G)$. 
Hence
 \begin{align*}
 \dim V^\mathfrak{X} (G)=(\dim V)|\hat{\mathcal{D}}|
 =n \sum_{j \in{\bar{\mathcal{D}}}}s_{j}
=n \sum_{j \in{\bar{\mathcal{D}}}}[\chi, 1_{G_j}].
\end{align*}
If $\mathfrak{X}$ is a linear representation of $G$, then $\dim V^\mathfrak{X}_{ij}(G)=1$ and the set
$$
\{ C_\mathfrak{X}(u_1\otimes e_{ij})~|~1\leq i\leq n,\  j\in \hat{\mathcal{D}}\}
$$ 
is an orthogonal basis of $V^\mathfrak{X}(G)$ (such representations of $G$ are called {\em o.b.-representations}). 

\section{\bf Induced Linear Operators on Generalized Cartesian Symmetry Classes}
Let $S_m$ be the full symmetric group of degree $m$, and let $G$ be a subgroup of $S_m$. Let $U$ be a unitary vector space. Given a linear operator $T: V \to V$, we can define the linear operator $\times^{m}T :\times^{m}V \to \times^{m}V$ by
$$
(\times^{m}T) v^\times =(Tv_1,\cdots,Tv_m),
$$
where $v^\times \in \times^{m}V$. It is easy to see that $T\to \times^{m}T$ is an algebraic homomorphism.
 Moreover, $(\times^{m}T)Q(\sigma)=Q(\sigma)(\times^{m}T)$ for any $\sigma\in G$, which implies that
$$
C_\mathfrak{X}(I \otimes \times^{m}T)=(I \otimes \times^{m}T)C_\mathfrak{X},
$$
and hence, $V^{\mathfrak{X}}(G)$ is an invariant subspace of $U \otimes \times^{m}V$ under the mapping $C_\mathfrak{X}$. We denote the restriction of $I \otimes \times^{m}T$ to $V^{\mathfrak{X}}(G)$ by $K^{\mathfrak{X}}(T)$ and call it an induced operator. Note that $T\to K^{\mathfrak{X}}(T)$ is also an algebraic homomorphism.


\begin{theorem}\label{P_1}
Suppose $\mathfrak{X}$  is an irreducible unitary representation of $G$ over unitary space $U$ and let $S, T \in End~(V)$ and $V^{\mathfrak{X}}(G) \neq {0}$. Then
\begin{itemize}
\item[(a)]
$K^{\mathfrak{X}}(T)=K^{\mathfrak{X}}(S) \iff T=S$, \\

\item[(b)]
$K^{\mathfrak{X}}(T)$ is invertible if and only if $T$ is invertible. \\
 \end{itemize}
 \end{theorem}
\begin{proof}
(a) Let $K^{\mathfrak{X}}(T)=K^{\mathfrak{X}}(S)$. Then for each $1\leq i\leq n$, $1\leq j\leq m$ and $u\in U$, we have
$$
K^{\mathfrak{X}}(T)(C_\mathfrak{X}(u_k \otimes e_{ij}))=K^{\mathfrak{X}}(S)(C_\mathfrak{X}(u_k \otimes e_{ij})).
$$
So
$$
C_\mathfrak{X}(u_k \otimes(\delta_{1j}Te_{i},\delta_{2j}Te_{i},\cdots , \delta_{mj}Te_{i}))=C_\mathfrak{X}(u_k \otimes (\delta_{1j}Se_{i},\delta_{2j}Se_{i},\cdots , \delta_{mj}Se_{i})).
$$
We put
$$
x_\ell=\delta_{\ell j}Te_{i}= \delta_{\ell j}\sum_{p=1}^n a_{pi}e_p=\sum_{p=1}^n \delta_{\ell j}a_{pi}e_p,
$$
$$
y_\ell =\delta_{\ell j}Se_{i}= \delta_{\ell j}\sum_{p=1}^n b_{pi}e_p=\sum_{p=1}^n \delta_{\ell j}b_{pi}e_p.
$$
Now we define two matrices $C$ and $D$ as follows:
$$ 
C=(C_{\ell p})=(\delta_{\ell j}a_{pi}),~~ D=(D_{\ell p})=(\delta_{\ell j}b_{pi}).
$$
Using Theorem \ref{Corollary 5.9}, we get
\begin{eqnarray}
Tr_\mathfrak{X}(CC^*) = Tr_\mathfrak{X}(CD^*)=Tr_\mathfrak{X}(DD^*). 
\end{eqnarray}
We can easily see that
\begin{eqnarray}
 (CC^*)_{\ell p} &=&\delta_{\ell j}\delta_{p j}\sum_{q=1}^{n}|a_{qi}|^2     \\
(DD^*)_{\ell p} &=&\delta_{\ell j}\delta_{p j}\sum_{q=1}^{n}|b_{qi}|^2    \\
 (CD^*)_{\ell p} &=&\delta_{\ell j}\delta_{p j}\sum_{q=1}^{n}a_{qi} \bar{b}_{qi} \\
 Tr_\mathfrak{X}(CC^*) &=&\sum_{q=1}^n|a_{qi}|^2 \sum_{\sigma \in G_j}\mathfrak{X}(\sigma)    \\
 Tr_\mathfrak{X}(DD^*) &=&\sum_{q=1}^n|b_{qi}|^2   \sum_{\sigma \in G_j}\mathfrak{X}(\sigma)    \\
 Tr_\mathfrak{X} (CD^*) &=&\sum_{q=1}^n a_{qi} \bar{b}_{qi} \sum_{\sigma \in G_j}\mathfrak{X}(\sigma).
\end{eqnarray}
Applying the trace map on Equations (1), (5), (6), (7),  we get
\begin{eqnarray}
\sum_{q=1}^{n}|a_{qi}|^2 \sum_{\sigma \in G_j}\chi(\sigma)
=\sum_{q=1}^{n}|b_{qi}|^2 \sum_{\sigma \in G_j}\chi(\sigma)
=\sum_{q=1}^{mn}{a_{qi} \bar{b}_{qi}} \sum_{\sigma \in G_j}\chi(\sigma).
\end{eqnarray}
If we choose $j\in \bar{D}$, then $ \sum_{\sigma \in G_j}\chi(\sigma)\neq {0}$.  Hence from Equation (8), we obtain
\begin{eqnarray}
\sum_{q=1}^{n}|a_{qi}|^2
=\sum_{q=1}^{n}|b_{qi}|^2
=\sum_{q=1}^{n}{a_{qi} \bar{b}_{qi}}~(1\leq i\leq n).
\end{eqnarray}
Thus
\begin{eqnarray}
\sum_{i=1}^{n}\sum_{q=1}^{n}|a_{qi}|^2
=\sum_{i=1}^{n}\sum_{q=1}^{n}|b_{qi}|^2
=\sum_{i=1}^{n}\sum_{q=1}^{n}{a_{qi} \bar{b}_{qi}},
\end{eqnarray}
which is equivalent to
$$
tr(A^*A) = tr(B^*B) = tr(B^*A),
$$
or
\begin{eqnarray}
||A||^2=||B||^2=<A, B>,
\end{eqnarray}
where $||.||$ is the Frobenius norm $\mathbb{C}_{n\times n}$.
From the equality condition in the Cauchy-Schwarz inequality, there exists a real number $\lambda$ such that  $A=\lambda B$.
Now by substituting in Equation (11), we get
$\lambda =1$ and then $A=B$. Therefore $T=S$.
The converse is obvious.\\

(b) If $T$ is invertible then $K^\mathfrak{X}(T)$ is invertible because $K^\mathfrak{X}$ is an algebraic homomorphism.\\
Conversely, if $K^\mathfrak{X}(T)$ is invertible then we prove that $T$ is invertible. To show this, suppose that $T$ is a singular operator. Then there exists a non-zero vector $e_{1}\in V$ such that $Te_{1}=0$. We can extend the set $\lbrace e_{1}\rbrace$ to an orthonormal basis $\lbrace e_1,\cdots,e_n\rbrace$ for $V$. Let $\lbrace u_1,\cdots,u_r\rbrace$ be an orthonormal basis for the unitary space $U$. Since $V^{\mathfrak{X}}(G)\neq 0$, we have $\bar{\mathcal{D}}\neq \emptyset$. Choose $j\in \bar{\mathcal{D}}$, then $j$ belongs to $\Omega=\cup_{1}^{r} \Omega_{k}$. Therefore, there exists $1\leq k\leq r$ such that $j\in \Omega_{k}$. It follows that $ C_\mathfrak{X}(u_k\otimes e_{ij})\neq 0$. Now we have
\begin{align*}
K^{\mathfrak{X}}(T)C_\mathfrak{X}(u_k \otimes e_{1j})=&(I\otimes \times^{m}T)C_\mathfrak{X}(u_k \otimes e_{1j})\\
=&C_\mathfrak{X}(u_k \otimes \times^{m}T(\delta_{1j}e_{1}, \cdots , \delta_{mj}e_{1}))\\
=&C_\mathfrak{X}(u_k \otimes(\delta_{1j}Te_{1}, \cdots , \delta_{mj}Te_{1}))\\
=&C_\mathfrak{X}(u \otimes(0, \cdots , 0))\\
=&0,
\end{align*}
which is a contradiction. Therefore, $T$ must be a non-singular operator. This completes the proof.
\end{proof}
\begin{theorem}\label{*}
Suppose $\mathfrak{X}$  is an irreducible unitary representation of $G$ over unitary space $U$ and let $S, T \in End~(V)$. Then $K^\mathfrak{X}(T)^\ast=K^\mathfrak{X}(T^\ast)$ and $K^\mathfrak{X}(T)$
is (a) normal, (b) unitary, (c) Hermitian, (d) skew-Hermitian, (e) p.s.d, or (f) p.d if and only if T has the
corresponding property.
 \end{theorem}
 \begin{proof}
 We know that $V^\mathfrak{X}(G)$ is an invariant subspace under of the both $I\otimes \times^{m}T$ and  $(I \otimes \times^{m}T)^\ast=I \otimes \times^{m}T^\ast$. 
Thus
$$
K^\mathfrak{X}(T)^\ast=((I\otimes \times^{m}T)~|_{V^{\mathfrak{X}}(G)})^\ast=(I\otimes \times^{m}T)^\ast~|_{V^{\mathfrak{X}}(G)}=K^\mathfrak{X}(T^*).
$$
If $T$ is (a) normal, (b) unitary, (c) Hermitian, (d) skew-Hermitian, (e) p.s.d, or (f) p.d, then $I\otimes \times^{m}T$ has
the corresponding property and so, $K^\mathfrak{X}(T)=(I\otimes \times^{m}T)~|_{V^{\mathfrak{X}}(G)}$ has also
the corresponding property.\\
Conversely, if $K^\mathfrak{X}(T)$ is normal, then
$$
K^\mathfrak{X}(T^\ast T)=K^\mathfrak{X}(T^\ast )K^\mathfrak{X}(T)=K^\mathfrak{X}(T)^\ast K^\mathfrak{X}(T)=K^\mathfrak{X}(T) K^\mathfrak{X}(T)^\ast=K^\mathfrak{X}(T) K^\mathfrak{X}(T^\ast)=K^\mathfrak{X}(T T^\ast).
$$
By Theorem \ref{P_1}, $TT^\ast=T^\ast T$, i.e., $T$ is normal. Similarly, if $K^\mathfrak{X}(T)$ is unitary
or Hermitian, then so does $T$. If $K^\mathfrak{X}(T)$ is skew-Hermitian, then $K^\mathfrak{X}(T)^\ast=-K^\mathfrak{X}(T)$, so $K^\mathfrak{X}(T^\ast)=K^\mathfrak{X}(-T)$ because $K^\mathfrak{X}$ is an algebraic homomorphism. Then, by Theorem \ref{P_1}, $T^\ast=-T$, i.e; $T$ is skew-Hermitian.\\
Now we consider $K^\mathfrak{X}(T)$ be a p.s.d operator. For every $v\in V$, define $v^\ast_j=(\delta_{1j}v, \cdots , \delta_{mj}v)$. Then  we have
\begin{align*}
\langle K^{\mathfrak{X}}(T)C_\mathfrak{X}(u_{1} \otimes v^\ast_{j}),C_\mathfrak{X}(u_{1} \otimes v^\ast_{j})\rangle \\
=&\langle C_\mathfrak{X}(u_1 \otimes \times^{m}T(\delta_{1j}v, \cdots, \delta_{mj}v)),
 C_\mathfrak{X}(u_1 \otimes (\delta_{1j}v, \cdots , \delta_{mj}v)) \rangle\\
=&\langle C_\mathfrak{X}(u_1 \otimes (\delta_{1j}Tv,\cdots ,\delta_{mj}Tv)), C_\mathfrak{X}(u_1 \otimes (\delta_{1j}v,\cdots ,\delta_{mj}v)) \rangle \\
=& \frac{1}{|G|}\langle Tr_\mathfrak{X}(A)u_1, u_1\rangle\\
=&\frac{1}{|G|}\langle \sum_{\sigma \in G}\mathfrak{X}(\sigma) \sum_{p=1}^n{a_{p \sigma (p)}}u_{1} ,u_{1} \rangle\\
=&\frac{1}{|G|}\langle \sum_{\sigma \in G}\mathfrak{X} (\sigma) \sum_{p=1}^n \delta_{pj} \delta_{\sigma (p)j} \langle Tv ,v\rangle  u_{1} ,u_{1}\rangle\\
=&\langle Tv ,v\rangle \frac{1}{|G|}\langle\sum_{ \sigma \in G} \mathfrak{X} (\sigma)\delta_{\sigma (j)j}u_1, u_1\rangle\\
=&\langle Tv ,v\rangle \langle \frac{1}{|G|}\sum_{ \sigma \in G_j} \mathfrak{X} (\sigma)u_1, u_1\rangle\\
=&\langle Tv ,v\rangle \frac{|G_{j}|}{|G|}\langle T_{j}u_1, u_1\rangle\\
=&\langle Tv ,v\rangle\frac{|G_{j}|}{|G|}\langle T_{j}u_1,T_{j} u_1\rangle\\
=&\langle Tv ,v\rangle\frac{|G_{j}|}{|G|}~||Tu_j||^2\geq 0,
\end{align*}
where $a_{pq}=\langle \delta_{pj}Tv, \delta_{qj}v\rangle=\delta_{pj}\delta_{qj}\langle Tv, v\rangle$.
Consequently, $\langle Tv ,v\rangle \geq 0.$ for all $v\in V$, i.e., $T$ is p.s.d.
If $K^\mathfrak{X}(T)$ is p.d, then $j\in \bar{\mathcal{D}}$. So $C_\mathfrak{X}(u_1 \otimes v^\ast_j)\neq 0$. Hence
\begin{align*}
\langle K^{\mathfrak{X}}(T)C_\mathfrak{X}(u_{1} \otimes v^*_{j}),C_\mathfrak{X}(u_{1} \otimes v^*_{j})\rangle=\langle Tv ,v\rangle\frac{|G_{j}|}{|G|}~||Tu_j||^2> 0.
\end{align*}
Consequently, $\langle Tv ,v\rangle >0$ for all $v\in V$, i.e., $T$ is p.d.
\end{proof}
\begin{corollary}
Suppose $\mathfrak{X}$  is an irreducible unitary representation of $G$ over unitary space $U$ and let $S, T \in End ~(V)$ such that $T\geq S$. Then
$K^\mathfrak{X}(T) \geq K^\mathfrak{X}(S)$.
\end{corollary}
\begin{proof}
Assume that $T\geq S$. Then $T-S\geq 0$. Hence $\times^{m}(T -S )\geq 0$. Therefore $\times^{m}T -\times^{m}S \geq 0$. Consequently 
$\times^{m}T \geq \times^{m}S$.
This implies that
$$
 K^\mathfrak{X}(T)=(I\otimes \times^m T)|_{V^\mathfrak{X} (G)} \geq(I\otimes \times^m S)|_{V^\mathfrak{X}(G)}=K^\mathfrak{X}(S).
 $$
\end{proof}

\begin{theorem}\label{rank}
Suppose $\mathfrak{X}$  is an irreducible unitary representation of $G$ over unitary space $U$ and let $ T \in End~(V)$. Then
$$
\text{rank}~(K^\mathfrak{X}(T))=\text{rank}~(T) ~|\hat{\mathcal{D}}|.
$$
 \end{theorem}
 \begin{proof}
 Suppose $\mathbb{F}=\lbrace u_1,\cdots,u_r\rbrace $ and $\mathbb{E}=\lbrace e_1,\cdots,e_n \rbrace $
are orthonormal bases for unitary
spaces $U$ and $V$, respectively. We can assume that the set $\{e_1, \cdots , e_s\}$ is a basis  for Ker~$T$. 
Then the set
$$
\{C_M(u_1\otimes e_{ij}) \mid   1\leq i\leq n, j\in\hat{\mathcal{D} } \}
$$
is a basis of $V^{\mathfrak{X}}(G)$.
For every $1\leq i\leq s$, we have
 \begin{align*}
K^{\mathfrak{X}}(T)C_\mathfrak{X}(u_1\otimes e_{ij})&=(I\otimes \times^{m}T) C_\mathfrak{X}(u_1\otimes e_{ij})\\
=&C_\mathfrak{X}(u_1 \otimes \times^{m}Te_{ij})\\
=&C_\mathfrak{X}(u_1 \otimes (\delta_{1j}Te_{i}, \cdots , \delta_{mj}Te_i)\\
=&C_\mathfrak{X}(u_1 \otimes(0, \cdots, 0))\\
=&0.
 \end{align*}
Let
$$
Te_i=\sum _{j=1}^{n}a_{ji}e_j\hspace{5mm} (1\leq i\leq n). 
$$
Let $s+1\leq i\leq n$ and $j\in \bar{\mathcal{D}}$. For every $1\leq k\leq m$, we define 
$$
x_k=\delta_{kj}Te_i=\delta_{kj}\sum _{\ell =1}^{n}a_{\ell i}e_{\ell}=\sum _{\ell =1}^{n}\delta_{kj}a_{\ell i}e_{\ell}.
$$
Let $B=[b_{k\ell }]=[\delta_{kj}a_{\ell i}]$. Then
 \begin{align*}
Tr_\mathfrak{X}(BB^*)=&\sum_{\sigma \in G}\mathfrak{X}(\sigma)\sum_{k=1}^{n}(BB^*)_{k\sigma(k)} \\
=&\sum_{\sigma \in G}\mathfrak{X}(\sigma) \sum_{k=1}^{n} \sum_{\ell =1}^{n}b_{k \ell}\bar{b}_{\sigma(k)\ell}\\
=&\sum_{\sigma \in G}\mathfrak{X}(\sigma) \sum_{k,\ell=1}^{n}\delta_{kj}a_{\ell i}\delta_{\sigma (k)j}\bar{a}_{\ell i} \\
=&\sum_{\ell=1}^{n}|a_{\ell i}|^2 \sum_{\sigma \in G_j}\mathfrak{X}(\sigma)\\
=&\sum_{\ell=1}^{n}|a_{\ell i}|^2 |G_j|T_{j}
 \neq 0.
 \end{align*}
Using Theorem \ref{Corollary 5.9}, we deduce that 
$
K^{\mathfrak{X}}(T)C_\mathfrak{X}(u_1\otimes e_{ij})=C_\mathfrak{X}(u_1\otimes x^{\times})\neq 0.
$
Also, for every $1\leq i\leq n$ and $j\in \bar{\mathcal{D}}$, we have
\begin{align*}
K^{\mathfrak{X}}(T)C_\mathfrak{X}(u_1\otimes e_{ij})=& C_{\mathfrak{X}}\left( u_1 \otimes \left( \delta_{1j}Te_{i}, \cdots , \delta_{mj}Te_{i}\right)\right) \\
 =&C_\mathfrak{X}\left( u_1 \otimes \left(\sum_{\ell=1}^{n}\delta_{1j}a_{\ell i}e_{\ell}, \cdots , \sum_{\ell=1}^{n}\delta_{mj}a_{\ell i}e_{\ell}\right)\right) \\
=&C_\mathfrak{X}\left(u_1 \otimes \sum_{\ell=1}^{n}a_{\ell i}(\delta_{1j}e_{\ell}, \cdots ,\delta_{mj}e_{\ell})\right) \\
=&C_\mathfrak{X}\left(u_1 \otimes \sum_{\ell=1}^{n}a_{\ell i}e_{\ell j}\right) \\
=&\sum_{\ell=1}^{n}a_{\ell i}C_\mathfrak{X}(u_1 \otimes e_{\ell j}).
\end{align*}
This shows that the representation of $K^\mathfrak{X}(T)$ under the basis
$$
\mathbb{S}=\bigcup_{j\in \hat{\mathcal{D}}}\left\{ C_\mathfrak{X}(u_1\otimes e_{1j}),\cdots,C_\mathfrak{X}(u_1\otimes e_{nj})\right\}
$$
is the following block matrix
\begin{center}
$\begin{bmatrix}
  {\begin{bmatrix} T \end{bmatrix}} 
       & \cdots & 0      \\
  \vdots & \ddots & \vdots \\ 
  0      & \cdots &   {\begin{bmatrix}
T
\end{bmatrix}}
\end{bmatrix}_{|\hat{\mathcal{D} }|\times |\hat{\mathcal{D} }|}$.
\end{center} 
Therefore 
$$
rank~ K^\mathfrak{X}(T)=rank~ (T) ~ |\hat{\mathcal{D}}|.
$$
\end{proof}
Using the above block matrix representation of $K^{\mathfrak{X}}(T)$, we obtain the following corollary.
\begin{corollary}
Let $\mathfrak{X}$ be an irreducible unitary representation of $G$ over unitary space $U$ and let $ T \in End~(V)$. Then
$$
det~K^{\mathfrak{X}}(T)=(det~T)^{|\hat{\mathcal{D}}|}.
$$
\end{corollary}


\section{\bf Open problems}
\begin{problem}
Characterize the subgroups of $S_m$ whose irreducible representations are all o.b.-representations.
\end{problem}
\begin{problem}
Let $G$ be a subgroup of $S_m$ and $\mathfrak{X}$ be an irreducible unitary representation of $G$. Determine the conditions on $\mathfrak{X}$ such that $V^{\mathfrak{X}}(G)$ has an orthogonal basis consisting the generalized Cartesian standard symmetrized vectors. 
\end{problem}
 \begin{problem}
Determine the conditions on $G$ and $\mathfrak{X}$ such that 
$V^{\mathfrak{X}}(G)\neq 0$.
\end{problem}


\begin{thebibliography}{99}

\bibitem{Babaei 1} E. Babaei and Y. Zamani, Symmetry classes of polynomials associated with the direct product of permutation groups, {\it Int. J. Group Theory}, {\bf 3}(4) (2014), 63--69.

\bibitem{Babaei 2}
E. Babaei, Y. Zamani and M. Shahryari, Symmetry classes of polynomials,
{\it Commun. Algebra}, {\bf 44} (2016), 1514--1530.

\bibitem{Gholami}
S. S. Gholami and Y. Zamani, Cartesian symmetry classes associated with certain groups, submitted.

\bibitem{Holmes 1}
R. R. Holmes and T.Y. Tam, Symmetry classes of tensors associated with certain groups, {\it Linear Multilinear Algebra}, {\bf 32}(1) (1992) 21-31.

\bibitem{Holmes 2} R. R. Holmes and A. Kodithuwakku, Orthogonal bases of Brauer symmetry classes of tensors for the dihedral group, {\it Linear Multilinear Algebra}, {\bf 61} (2013) 1136 -1147.

\bibitem{Isaacs}
M. Isaacs, {\em  Character Theory of Finite Groups}, Academic Press, 1976.

\bibitem{Lei}
T. G. Lei, {Notes on Cartesian symmetry classes and generalized trace functions}, {\it Linear Algebra Appl.}, {\bf 292} (1999) 281-288.

\bibitem{Lei2}
T. G. Lei, {Generalized Schur functions and generalized decomposable symmetric tensors}, {\it Linear Algebra Appl.}, {\bf 263} (1997) 311-332.

\bibitem{Merris}
R. Merris, {\em Multilinear Algebra}, Gordon and Breach Science Publishers, 1997.

\bibitem{Zamani} Y. Zamani and M. Shahryari, On the dimensions of Cartesian symmetry classes, {\it  Asian-Eur. J. Math.},
 {\bf 5} (3) (2012) Article ID 1250046 (7 pages).
 
\bibitem{Zamani 1} Y. Zamani and E. Babaei, The dimensions of cyclic symmetry classes of polynomials, {\it J. Algebra Appl.}, {\bf 13} (2014) 3318--3321.

 
\end{thebibliography}
\end{document}